\documentclass[a4paper,11pt,headings=standardclasses]{article}%
\title{Magnetic geodesics, Hodge Laplacian eigenvalues, and isoperimetric inequalities}

\usepackage[T1]{fontenc}
\usepackage{comment}
\usepackage{microtype}
\usepackage{pinlabel}
\usepackage{float}
\usepackage[hypertexnames = false]{hyperref}
\usepackage{subfig}
\usepackage{bbm}
\usepackage{amsmath, amssymb,amsthm,thmtools,cite,mathtools,setspace,enumitem,tikz-cd}

\usepackage{thm-restate}
\usepackage{bm}
\usepackage[T1]{fontenc} 
\usepackage{tikz-cd}

\usepackage{AlegreyaSans}
\usepackage[osf]{cochineal} 
\usepackage[scaled=.95]{cabin} 
\usepackage[utopia,vvarbb]{newtxmath}

\makeatletter
\newcommand\RedeclareMathOperator{%
	\@ifstar{\def\rmo@s{m}\rmo@redeclare}{\def\rmo@s{o}\rmo@redeclare}%
}
\newcommand\rmo@redeclare[2]{%
	\begingroup \escapechar\m@ne\xdef\@gtempa{{\string#1}}\endgroup
	\expandafter\@ifundefined\@gtempa
	{\@latex@error{\noexpand#1undefined}\@ehc}%
	\relax
	\expandafter\rmo@declmathop\rmo@s{#1}{#2}}
\newcommand\rmo@declmathop[3]{%
	\DeclareRobustCommand{#2}{\qopname\newmcodes@#1{#3}}%
}

\@onlypreamble\RedeclareMathOperator
\makeatother

\DeclareSymbolFont{bbold}{U}{bbold}{m}{n}

\newtheorem{thm}{Theorem}

\numberwithin{thm}{section}

\newtheorem{lem}[thm]{Lemma}

\newtheorem{prop}[thm]{Proposition}

\newtheorem{cor}{Corollary}
\declaretheorem[style=remark]{remark}

\DeclareMathOperator{\scl}{{scl}}

\DeclareMathOperator{\vol}{vol}
\DeclareMathOperator{\inj}{inj}

\DeclareMathOperator{\area}{area}

\renewcommand{\H}{{\mathbb H}}

\newcommand{\R}{{\mathbb R}}

\newcommand{\cA}{\mathcal{A}}

\newcommand{\cG}{\mathcal{G}}\newcommand{\cH}{\mathcal{H}}

\newcommand{\cS}{\mathcal{S}}

\newcommand{\e}{\varepsilon}

\newcommand{\g}{\gamma}
\renewcommand{\a}{\alpha}
\renewcommand{\L}{\Lambda}

\newcommand{\fm}{\mathfrak m}

\newcommand{\norm}[1]{\left\lVert#1\right\rVert}

\renewcommand{\d}{ \partial}

\definecolor{darkmidnightblue}{rgb}{0.0, 0.2, 0.4}
\hypersetup{
    colorlinks=true,
    linkcolor=darkmidnightblue,
    citecolor=darkmidnightblue,
    filecolor=darkmidnightblue,
    urlcolor=darkmidnightblue
}

\author{Cameron Gates Rudd}

\DeclareMathOperator{\sarea}{sarea}

 \date{}
\begin{document}
\maketitle

\begin{abstract}
An isoperimetric constant relating length and stable area, or alternatively for hyperbolic manifolds, length and stable commutator length, serves as a Cheeger constant for the smallest eigenvalue of the Hodge Laplacian acting on coexact 1-forms. Using properties of the magnetic geodesic flow associated to the differential of a coexact eigenform, and its behavior at Mañé's critical energy level, we give new proofs of these Cheeger-like inequalities, with improved (and explicit) constants and volume dependence. We also make a few observations about the relationship between Mañé's critical values and the eigenvalues, when the manifold is hyperbolic.
\end{abstract}

The smallest eigenvalue of the Hodge Laplacian acting on coexact 1-forms is a fundamental, but difficult to grasp, analytic invariant of a Riemannian manifold. Recently, there has been considerable interest in understanding the behavior of these eigenvalues geometrically, particularly for hyperbolic 3-manifolds, where a large coexact spectral gap is related to nonexistence of irreducible solutions to the Seiberg-Witten equations and to torsion in homology. See \cite{BSV, LipnowskiStern, Rudd, BoulangerCourtois,LinLipnowski, bootstrap, princeton-team, Zung, LL2}  for recent work in these directions. 

In this paper, we refine and give new proofs of analogues of Cheeger's inequality for the Hodge Laplacian acting on coexact 1-forms. In Theorems \ref{thm:intro} and \ref{thm:intro:area}, we relate the coexact spectral gap to isoperimetric constants that measure the complexity of filling loops by surfaces. Additionally, for hyperbolic 3-manifolds, we obtain explicit constants in these theorems. Explicit constants for such estimates have applications; see for instance \cite{Lin}.
Our methods are very different from those used in previous such estimates (see \cite{LipnowskiStern, Rudd, BoulangerCourtois, Zung}), which allows for much better control over the constants appearing. 

Our approach is to apply the weak KAM theory to a flow associated to an eigenform of the Laplacian (or in the three dimensional case, the curl operator).
The differential of such an eigenform, viewed as a magnetic field, defines a deformation of the geodesic flow, whose orbits are called magnetic geodesics. By relating various quantities and geometric objects that naturally arise in studying these flows, we obtain our main theorems. Of particular importance are Mañé's critical values, which record energy levels at which the dynamics of these flows change.

   For hyperbolic manifolds with trivial first Betti number, we consider the stable isoperimetric constant $\rho(M)$, introduced in \cite{LipnowskiStern}, given by $$\rho(M):=\inf_{\substack{\g\neq1\\ [\g]=0}}\frac{\ell(\gamma)}{\scl(\g)},$$ where the infimum runs over all nontrivial rationally nullhomologous closed geodesics. This constant compares the (algebraic/topological) stable commutator length function $\scl$, defined for rationally nullhomologous elements of the fundamental group, with the (geometric) hyperbolic geodesic length function $\ell$. We note that as $\scl$ is really a function on the fundamental group of such a manifold, by using the holonomy representation to compute lengths, this constant can be obtained from the fundamental group.

Our first result is the following; compare with \cite{LipnowskiStern} and \cite{Rudd}.

 \begin{thm}[cf. Theorem \ref{thm:main}]\label{thm:intro}
         Let $M$ be a closed oriented hyperbolic $n$-manifold with $b_1(M)=0$ and $\inj(M)>r.$ Let $\lambda^*$ be the smallest eigenvalue of the Hodge Laplacian acting on coexact 1-forms and assume $\lambda^*<1$. Then either $$\frac{\rho(M)}{8\pi C_r\sqrt{\vol(M)}}\leq\sqrt{\lambda^*},$$ or else, $$\frac{1}{8C_r^2\vol(M)}\leq \sqrt{\lambda^*},$$ where  $C_r$ is the constant from the mean-value-inequality (see Proposition \ref{prop:MV} or Theorem \ref{thm:appendix}).
\end{thm}

    We obtain explicit expressions for the constant $C_r$ above in dimension 3. We state the $r=1$ case as the following: 
  
 \begin{thm}\label{thm:intro3man}
         Let $M$ be a closed oriented hyperbolic $3$-manifold with $b_1(M)=0$ and $1< \inj(M).$ Let $\lambda^*$ be the smallest eigenvalue of the Hodge Laplacian acting on coexact 1-forms and assume $\lambda^*<1$. Then either $$\frac{\rho(M)}{9\pi \sqrt{\vol(M)}}\leq\sqrt{\lambda^*},$$ or else, $$\frac{1}{9\vol(M)}\leq \sqrt{\lambda^*}.$$
\end{thm}

The proof of Theorem \ref{thm:intro} identifies additional dynamical structure.
In particular,  the weak KAM theory produces distinguished magnetic geodesics at Mañé's critical energy level; these are the magnetic geodesics that comprise the Aubry set. These distinguished magnetic geodesics are analogous to stretch laminations associated to cohomology classes, as in \cite{DaskalopoulosUhlenbeck, FathiSiconolfi}\footnote{Understanding the structure of the Aubry set in explicit examples would be very interesting. Explicit examples of cohomology stretch laminations in dimensions 2 and 3 respectively are studied in  \cite{FarreLandesbergMinsky,RuddStretch}.}.  When the manifold is hyperbolic and the corresponding eigenvalue is sufficiently small, these magnetic geodesics closely shadow hyperbolic geodesics, and this "closeness" is in fact controlled by the eigenvalue of the eigenform\footnote{With the geodesic cohomology stretch laminations of \cite{DaskalopoulosUhlenbeck, FathiSiconolfi} in mind, one can view this as a manifestation of the idea that small eigenvalue coexact forms are "almost" harmonic, and thus almost are cocycles.}.

In this small eigenvalue regime, the proof relates periods over these shadowing geodesics to stable commutator length and the size of the eigenvalue using Bavard duality via the quasimorphism associated to an eigenform.

Our next result uses the Hamiltonian and Lagrangian descriptions of Mañé's strict critical value to prove a general lower bound on the coexact spectral gap. A consequence of this is a new, short, proof of the main theorem of \cite{BoulangerCourtois}. In their theorem, a mean-value-inequality constant appears, coming from a mean-value-inequality of Gallot from \cite{Gallot} that requires uniformly bounded diameter and curvature bounds. Using Gallot's estimate in Theorem \ref{thm:intro:area} below then recovers their theorem. We also obtain an explicit version of this estimate for hyperbolic 3-manifolds. 

 In Theorem \ref{thm:intro:area}, the isoperimetric constant we consider is the degree one Cheeger constant, defined by \cite{BoulangerCourtois},  $$h_1(M) = \inf_{\a}\frac{\ell(\alpha)}{\sarea(\alpha)},$$ where the infimum runs over all immersed rationally nullhomologous curves (including nullhomotopic ones), and $\sarea(\alpha)$ is the stable area, which measures the minimal (normalized) area of a surface bounding a multiple of $\alpha.$

 \begin{thm}[cf. Theorem \ref{thm:cheeger}]\label{thm:intro:area}
     Let $M$ be a closed oriented Riemannian $n$-manifold. Let $\lambda^*$ be the smallest eigenvalue of the Hodge Laplacian acting on coexact 1-forms. Assume the differential of a $\lambda^*$ eigenform  satisfies a mean-value-inequality with constant $C$. Then $$\frac{h_1(M)}{C\sqrt{\vol(M)}}\leq \sqrt{\lambda^*}.$$
 \end{thm}

For hyperbolic 3-manifolds, the constant $C$ can be made explicit. For instance, we have:
 
 \begin{thm}\label{thm:intro:area-hyp}
     Let $M$ be a closed oriented hyperbolic $3$-manifold with injectivity radius $1<\inj(M)$.Let $\lambda^*$ be the smallest eigenvalue of the Hodge Laplacian acting on coexact 1-forms and assume $\lambda^*<1$. Then $$\frac{h_1(M)}{1.04\sqrt{\vol(M)}}\leq \sqrt{\lambda^*}.$$
 \end{thm}

 \begin{remark}
         For some small volume examples, Lin and Lipnowski in \cite{LinLipnowski} (see their Table 2) explicitly computed bounds on the smallest coexact eigenvalue using the trace formula. For example:
             Let $M_{34}$ be the 34th manifold in the {\tt SnapPy} census, which Lin and Lipnowski study in detail. Then \begin{align*}
                 &\vol(M_{34}) < 1.91222, \\ &\inj(M_{34})> 0.24958\\
                 &\sqrt{\lambda^*}\in  (0.03619, 0.07329).
             \end{align*}  
             With this injectivity radius lower bound, one can take the mean-value-inequality constant to be $$C= 4.06.$$
             Consequently, we have a completely explicit upperbound on the degree one Cheeger constant for this hyperbolic 3-manifold: $$h_1(M_{34})\leq 0.41148.$$ 
 \end{remark}

The dynamical perspective also allows us to convert information from spectral geometry into information about the universal Mañé critical value and Mañé's strict critical value. For hyperbolic 3-manifolds, the curl operator enables a particular clean estimate relating the universal cover Mañé critical value to the coexact spectral gap.
 
\begin{prop}[cf. Proposition \ref{lem:mane-eigval-ineq}]  Let $M$ be a closed hyperbolic 3-manifold and let $\omega$ be a coexact curl eigenform with eigenvalue $\mu$ and $\norm{\omega}_\infty=1$. Then $$\sqrt{2\tilde c(\omega)}\leq|\mu|,$$ where $\tilde c(\omega)$ is the universal cover Mañé's critical value for the magnetic geodesic flow associated to $\omega.$
\end{prop}

When the smallest eigenvalue $\mu$ satisfies  $|\mu| \ll 1/\vol(M),$ Mañé's universal cover critical value is much smaller than the strict critical value.
Using this, the examples of Theorem C in \cite{Rudd} produce bounded geometry hyperbolic manifolds $M_n$ and eigenforms $\omega_n$ such that the size of their eigenvalues, and thus the universal critical values $\tilde c(\omega_n)$, decay exponentially fast in the volume, while the strict $c_0(\omega_n)$ decays at most linearly in the volume.

\begin{prop}[cf. Proposition \ref{prop:examples}]\label{intro:Examples}
    There exists a sequence $M_n$ of closed hyperbolic 3-manifolds with $\inj(M_n)>r$ and volume $V_n=\vol(M_n)\to\infty$ along with curl eigenforms $\omega_n\in\Omega^1(M_n)$ normalized so that $||\omega_n||_\infty=1$ and  such that for constants $c,C>0$, the following holds $$\frac{\tilde c(\omega_n)}{c_0(\omega_n)}\leq CV_n^3e^{-2cV_n}.$$
\end{prop}

\section{Eigenforms and magnetic fields}\label{sec:mag-fields}
\subsection{Hodge and de Rham complex}
Let $(M,g)$ be a closed oriented Riemannian $n$-manifold. Associated to the Riemannian metric is the Hodge star $\star:\Omega^p(M)\to\Omega^{n-p}(M)$, which in turn defines an inner product via $$\langle \omega,\eta\rangle = \int_M\omega\wedge \star\eta = \int_M\langle \omega_x,\eta_x\rangle_x d\vol(x),$$ where $\langle\cdot,\cdot\rangle_x$ is the pointwise inner product on the cotangent bundle induced from the metric $g.$   This defines an $L^2$-norm on $\Omega^p(M)$: $\norm{\omega}_2^2 = \langle \omega,\omega\rangle.$ The pointwise norm $|\omega|_x = \langle \omega_x,\omega_x\rangle_x^{1/2}$ leads alternatively to a uniform norm $$\norm{\omega}_\infty = \sup_{x\in M} |\omega_x|_x.$$
The inner product also determines the Hodge Laplacian $$\Delta_p = (d+d^*)^2:\Omega^p(M)\to\Omega^p(M),$$ where $d^* =(-1)^{n(p+1)+1} \star d \star$  is the adjoint of $d.$ We drop the subscript $p$ in the Laplacian notation when it is clear.
The Hodge theorem decomposes $\Omega^p(M)$ as an orthogonal sum $$\Omega^p(M) =\cH^p(M)\oplus d\Omega^{p-1}(M)\oplus d^*\Omega^{p+1}(M).$$ The summand $\cH^p(M) = \ker(\Delta_p)$ consists of harmonic forms. The other summands are the exact forms $\omega \in d\Omega^{p-1}(M)$ and coexact forms $\omega\in d^* \Omega^{p+1}(M).$ A differential form $\omega$ such that $\Delta\omega=\lambda \omega$ is called an eigenform with eigenvalue $\lambda$. 

The summands in the Hodge decomposition of the de Rham complex further decompose as sums of eigenspaces of the Laplacian. The smallest eigenvalue of the Laplacian restricted to the space of coexact degree $p$-forms is called the first coexact eigenvalue in degree $p$. In the rest of this paper, we focus on the first coexact eigenvalue for the Laplacian acting on 1-forms. Thus when there will be no confusion, we refer to this eigenvalue as the first coexact eigenvalue, or as the coexact spectral gap, without reference to degree.

\begin{remark}
    Some of our results are particularly suited to dimension three, due to the curl operator $\star d.$ Coexact curl eigenforms are coexact 1-forms $\omega$ such that $$\star d \omega= \mu\omega$$ for some $\mu\in\R.$
The eigenvalues of the curl operator are related to the eigenvalues of the Hodge Laplacian acting on coexact 1-forms: the square of every curl eigenvalue is a Laplace eigenvalue, and for every coexact Laplace eigenvalue, a square root appears as a curl eigenvalue. Mostly, we will work with Laplace eigenforms to avoid handling two cases, but some estimates are cleaner in the curl case, and this is most relevant in Section \ref{sec:uni-crit-vel}.
\end{remark}

We say a $p$-form $\omega$ satisfies a $C$-mean-value-inequality if $$\norm{\omega}_\infty\leq C\norm{\omega}_2.$$
Eigenforms with bounded eigenvalue satisfy a uniform $C$-mean-value inequality, where the constants depend on the local geometry of the manifold. We highlight one such estimate from the literature. In Section \ref{sec:appendix}, and Theorem \ref{thm:appendix} in particular, we give another proof with explicit constants in dimension 3. 

\begin{prop}[Proposition 2.2 \cite{LipnowskiStern}]\label{prop:MV} Let $M$ be a closed hyperbolic $n$-manifold with injectivity radius $\inj(M)>r.$ Let $\omega$ be a degree $p$-eigenform of the Hodge Laplacian with eigenvalue at most $\L>0$. Then there is a constant $C(n,r,\L)$ such that $\omega$ satisfies a $C(n,r,\L)$-mean-value-inequality.
\end{prop}
Note that if $\omega$ is a coexact eigenform in degree $p$ with eigenvalue $\mu^2,$ the $p+1$ form $d\omega$ is an exact eigenform with the same eigenvalue.

\subsection{Magnetic fields and Mañé's critical values}
Let $(M,g)$ be a closed oriented Riemannian $n$-manifold.
Given a 1-form $\omega\in\Omega^1(M),$ its differential is a 2-form $\Omega = d\omega$, which defines a magnetic field on $M$.
The corresponding Hamiltonian of a covector $p\in T_x^*M$ is given by $$H(x,p) = \frac{1}{2}|p+\omega_x|_x^2;$$ where we use the dual metric on the cotangent bundle.
The corresponding Lagrangian on the tangent bundle is given by $$L(x,v) = \frac{1}{2}g(v,v)_x - \omega_x(v).$$ The associated magnetic geodesic flow is the Euler-Lagrange flow of this Lagrangian, or the dual of the Hamiltonian flow of $H$ (using the canonical symplectic structure on $T^*M$).

Associated to the magnetic field $\Omega$ is a bundle endomorphism $Y$, called the Lorentz force, defined via the metric by the equation $$ \Omega(u,v) = g( Y(u),v).$$ The orbits of the magnetic flow are parametric integral curves $\gamma$ such that $$\nabla_{\dot\gamma}\dot\gamma = Y(\dot\gamma).$$ When $\Omega=0$ (i.e., when $\omega$ is closed), this is just the geodesic equation. Thus we can view the magnetic flow as describing a deformation of the usual geodesic flow; this is key to our analysis in the hyperbolic case. 

The speed of magnetic geodesics is constant along trajectories, but unlike regular geodesics, the solutions to the magnetic geodesic equation with different speeds are not simply reparameterizations; different curves appear as solutions with different speeds. 
The dynamics of magnetic geodesics depends on the energy level. Various phase transitions occur at critical energy levels. 

We first define critical values separately for the Hamiltonian and the Lagrangian; it turns out these are the same \cite{CIPP}. The Lagrangian definition requires the notion of action of a path.
Define the space of time $T$ parametrized paths from a point $x\in M$ to $y\in M$ to be $$\Pi_T(x,y) = \{\alpha:[0,T]\to M~:~ \alpha(0)=x,~\alpha(T)=y,~\alpha \in C^{Lip}([0,T],M)\}$$
and let $$\Pi(x,y) = \sqcup_{T>0}\Pi_T(x,y).$$  
Define also $$\Pi(M) = \sqcup_{x\in M} \Pi(x,x).$$
For $\alpha\in\Pi(x,y),$ let $T(\alpha)$ be the $T$ such that $\alpha\in \Pi_{T(\alpha)}(x,y).$
The action of a path $ \alpha\in\Pi(x,y)$ is given by $$A(\alpha)=\int_0^{T(\alpha)}L(\alpha(t),\dot\alpha(t))dt.$$

The Lagrangian Mañé critical value is given by $$c(L)=\inf\{k\in\R~:~ \forall \alpha\in \Pi(M),~ A(\alpha)+kT(\alpha)\geq0\},$$ 
and the Hamiltonian Mañé critical value is given by $$c(H) = \inf_{u:M\to\R} \frac{1}{2}||du+\omega||_\infty^2,$$ where the infimum is over smooth maps.

The function $\cS_k:\Pi(M)\to\R$ given by $\cS_k(\alpha) = A(\alpha)+kT(\alpha)$ is called the free-period action functional; see \cite{Abbondandolo}, and Section 4 in particular, for a focused reference. 

The two notions of critical value turn out to agree.

\begin{thm}[Theorem A \cite{CIPP}]\label{thm:CIPP} Let $M$ be an orientable Riemannian $n$-manifold covering a closed orientable Riemannian $n$-manifold $M_0$. Let $\omega\in\Omega^1(M_0)$ be a 1-form; consider the pullback of $\omega$ to $M$ and the corresponding Hamiltonian $H$ and Lagrangian $L$ for $M$. Then $$c(L)=c(H).$$ 
\end{thm}

With this equality in mind, we will write $c(\omega)$ to denote the critical value above for the system on $M$ defined by the 1-form $\omega$. 

A slight modification of the definitions above leads to the Mañé strict critical values.
From the Hamiltonian perspective, this amounts to allowing one to also minimize over harmonic forms when defining the critical value:
$$c_0(H) = \inf_{\substack{u:M\to\R\\ h\in\cH^1(M)}}\frac{1}{2}||du+h+\omega||_\infty^2.$$

For the Lagrangian definition, consider the map $\Pi(M)\to H_1(M;\R)$, where one views elements of $\Pi(M)$ as Lipschitz singular 1-cycles. Let $\Pi_0(M)$ consist of all loops in $\Pi(M)$ which map to trivial 1-cycles in $H_1(M;\R).$ Then, the Lagrangian definition of the strict Mañé critical value is $$c_0(L) = \inf\{k\in\R~:~\forall \alpha\in \Pi_0(M),~A(\alpha)+kT(\alpha)\geq0\}.$$

As in the case of $c(\omega),$ the two definitions of strict critical value agree; this follows from results in \cite{CIPP} and \cite{PP}.  We therefore write  $c_0(\omega):=c_0(H)=c_0(L)$ for the strict critical values.

 \begin{thm}[Corollary 1 and subsequent remark \cite{CIPP}] 
     Let $(M,g)$ be a closed oriented Riemannian $n$-manifold and let $\omega\in \Omega^1(M)$ be a 1-form. Then $c_0(H)=c_0(L).$
 \end{thm}

The following consequence of the Hamiltonian characterization of the strict critical value will be useful. 
\begin{lem}\label{lem:norm-comp} Let $(M,g)$ be a closed oriented Riemannian $n$-manifold and
    consider $ \sqrt{2 c_0(\omega)}$ the critical speed for the magnetic flow defined by a coexact 1-form $\omega$. Then $$
||\omega||_2\leq 
\sqrt{\vol(M)} \sqrt{2 c_0(\omega)}.$$
 \end{lem}
 \begin{proof}
 Take $u_i$ and $h_i$ a sequence of functions and harmonic 1-forms respectively realizing the minimum defining the strict critical value. The Hodge decomposition implies $\norm{\omega}_2\leq ||h_i+du_i+\omega ||_2,$ and then an $L^\infty$-$L^2$-norm comparison implies $\norm{\omega}_2\leq \sqrt{\vol(M)} ||h_i+du_i+\omega||_\infty$. Taking $i\to\infty,$ the right-hand-side tends to $\sqrt{\vol(M)} \sqrt{2 c_0(\omega)}$, as desired. 
 \end{proof}

The following is related to Theorem A in \cite{BurnsPaternain} (and is in fact an equality).

 \begin{lem}\label{lem:null-loop-id}
 Let $M$ be a closed oriented Riemannian $n$-manifold. Let $\omega$ be a 1-form. 
     There is an inequality $$\sqrt{2c_0(\omega)} \leq \sup\limits_{\substack{\g:S^1\to M\\ [\g]=0\in H_1(M;\R)}} \frac{1}{\ell(\gamma)}\int_\gamma\omega.$$
 \end{lem}

 \begin{proof}
 First note that if $c_0(\omega)=0,$ this is trivial, so we assume $c_0(\omega)>0.$
From the definition of the critical value using the free period action, we have that for all $k< c_0(\omega)$, there exists a nullhomologous loop $\g\in\Pi_0(M)$ with period $T=T(\g)$ such that the free period action satisfies $\cS_k(\g)<0.$ Because the constant speed parametrization will minimize the action, we can assume $\g$ is parametrized to have constant speed $s=\ell(\g)/T$. We compute

\begin{align*}
0> \cS_k(\g) &= \int_0^{T}\frac{1}{2}|\dot\g(t)|^2-\omega(\dot\g(t))dt + kT
=\frac{\ell(\g)^2}{2T}-\int_{\g}\omega + kT.
\end{align*}

We can further assume that $T = \ell(\g)/ \sqrt{2k}$, as this minimizes $\ell(\g)^2/(2T)+kT$. In particular, because $k<c_0(\omega),$ the above expression must be negative for this value of $T$.
Thus $$0> \cS_k(\g) = \frac{\ell(\gamma)}{2(2k)^{-1/2}} - \int_\g \omega + k\frac{\ell(\g)}{\sqrt{2k}}  .$$
Rearranging and dividing by $\ell(\g)$, we find $$\frac{1}{\ell(\g)}\int_\g\omega > \frac{\sqrt{2k}}{2}+\frac{\sqrt{k}}{\sqrt{2}} = \sqrt{2k}.$$
Now, take a sequence $k_i\to c_0(\omega)$ with $k_i<c_0(\omega)$ and find nullhomologous loops $\g_i$ as above for each such $k_i$. Then this estimate implies $$\lim\limits_{i\to\infty}\frac{1}{\ell(\g_i)}\int_{\g_i}\omega\geq \sqrt{2c_0(\omega)},$$ and the claim follows.
\end{proof}

Define the stable area of a rationally nullhomologous loop $\g$ to be $$\sarea(\g)= \inf\left\{\frac{\area(S_n)}{n}~:~S_n\to M,~\d S_n = \g^n,~n>0\right\}.$$

Recall that the degree one Cheeger constant is defined to be $$h_1(M) = \inf_{\a}\frac{\ell(\alpha)}{\sarea(\alpha)},$$ where the infimum runs over all immersed rationally nullhomologous curves (including nullhomotopic ones).

We can now prove the analogue of Cheeger's inequality, refining the main theorem in \cite{BoulangerCourtois} and the lower bound for 1-forms in Theorem 3.4 in \cite{Zung}.

 \begin{thm}[cf. Theorem \ref{thm:intro:area}]\label{thm:cheeger}
     Let $M$ be a closed oriented Riemannian $n$-manifold and let $\omega$ be a coexact eigenform with eigenvalue $\mu^2$ of the Laplacian, normalized so that $\norm{\omega}_2=1$. Assume $d\omega$ satisfies a $C$-mean-value-inequality. Then $$h_1(M)\leq C\sqrt{\vol(M)}\mu.$$
 \end{thm}
 \begin{proof}
      Stokes' theorem and the definition of stable area imply that for any rationally nullhomologous loop $\g $ and surface $S_n$ with $\d S_n=\g^n$,  we have $$\frac{1}{\ell(\g)}\int_\g\omega =  \frac{1}{n\ell(\g)}\int_{\g^n}\omega= \frac{1}{n\ell(\g)}\int_{S_n} d\omega \leq \frac{||d\omega||_\infty \area(S_n)}{n\ell(\gamma)}.$$
      Taking a sequence of surfaces $S_n$ with boundary $\gamma^n$ such that $\area(S_n)/n\to \sarea(\g),$ we find  
      
      $$\frac{1}{\ell(\g)}\int_\g\omega  \leq  \sup_{\substack{\g :S^1\to M \\ [\g]=0}}\frac{\sarea(\g)}{\ell(\g)}||d\omega||_\infty,$$ where the supremum runs over all nullhomologous loops in $M.$ Using the mean value inequality, our normalization of $\omega,$ and Lemma \ref{lem:null-loop-id}, we find $$\sqrt{2c_0(\omega)}\leq  \sup\limits_{\substack{\g:S^1\to M\\ [\g]=0\in H_1(M;\R)}} \frac{1}{\ell(\gamma)}\int_\gamma\omega\leq \sup_{\substack{\g :S^1\to M \\ [\g]=0}} \frac{\sarea(\g)}{\ell(\g)}C\mu = h_1(M)^{-1}C\mu.$$  Using Lemma \ref{lem:norm-comp}, we have $$||\omega||_2\leq \sqrt{\vol(M)}\sqrt{2c_0(\omega)} .$$ We can therefore conclude $$\frac{h_1(M)}{C\sqrt{\vol(M)}}\leq\mu. \qedhere$$
 \end{proof}

  \subsection{Stable isoperimetric constants and magnetic comass}
To prove Theorem \ref{thm:main} relating the coexact spectral gap to stable commutator length for hyperbolic manifolds, we need more refined information about magnetic geodesics. 
Fix a closed Riemannian manifold $M$ and a 1-form $\omega\in\Omega^1(M).$
Denote by $\cG_s(\omega)$ the set of magnetic geodesic trajectories with speed $s.$ Note that for $\eta$ a closed 1-form, $\cG_s(\eta)$ is the set of speed $s$ geodesic trajectories.
Define the magnetic comass with speed $s>0$ to be $$\fm_s(\omega):=\sup\limits_{\g\in\cG_s(\omega)}A_\g\omega ,$$ where $A_\g\omega$ is defined to be $$A_\g\omega=\limsup_{T\to \infty}\frac{1}{sT}\left|\int_{0}^T\omega(\dot\gamma(t))dt\right|,$$ which is well defined since $\omega$ is bounded. Note that when $\omega$ is closed, this is just the usual comass of the cohomology class of $\omega.$

Define the geodesic comass of $\omega$ to be $$\fm_\infty(\omega) = \sup_{\g\in \cG_s(0)} A_\g\omega,$$ for any $s>0.$
Note that when $M$ has Anosov geodesic flow, the Anosov closing lemma implies $$\fm_\infty(\omega)= \sup_{\g \text{ closed geodesic}} \frac{1}{\ell(\gamma)}\left|\int_\g\omega\right|.$$
Note also that when the geodesic flow is Anosov,  the quasigeodesics converge to geodesics  as the speed $s\to\infty,$ (see Section \ref{sec:hyp-mag-geods}), and one obtains $$\fm_\infty(\omega)=\lim\limits_{s\to\infty}\fm_s(\omega).$$

The magnetic comass at the strict critical speed is closely related to the norm of $\omega$, and moreover, the weak KAM theory produces specific calibrated orbits with critical speed; this is what drives our proofs.
Let $H$ be the Hamiltonian associated to a 1-form $\omega.$ Consider the equation $$H(x, p)= \frac{1}{2}|\omega_x+p|^2=c.$$ A Lipschitz function $u:M\to\R$ (recall that a Lipschitz map is almost everywhere differentiable) is a subsolution of this equation if for almost all points $x\in M$, one has $$H(x,du_x)\leq c.$$
There exist subsolutions for all $c\geq c(\omega)$, where $c(\omega)$ is Mañé's critical value, and these subsolutions can be taken to be smooth when $c>c(\omega)$.  The existence of $C^1$ subsolutions at the critical value was proved by Fathi-Siconolfi \cite{FathiSiconolfi}, this was then improved to $C^{1,1}$ by \cite{Bernard}. Subsolutions for $c= c(\omega)$ are called critical subsolutions. 

We now follow the treatment in \cite{Bernard}; note that the arXiv version of \cite{Bernard} is updated compared to the publication version, our references here are therefore to the arXiv version \cite{BernardArXiv}.

The projected Aubry set $\cA$ is defined to be the set of points $x \in M$ for which given any critical subsolution $u:M\to\R$ differentiable at $x$, the Hamiltonian of $du$ at $x$ is the critical value: $$\cA(\omega)=\{x\in M~:~ H(x,du_x)=c(\omega) \text{ for every critical subsolution } u:M\to\R\}.$$ 
Let $u:M\to\R$ be any $C^1$ critical subsolution, define the Aubry set $$\widetilde \cA(\omega)=\{(x,du_x)~:~x\in\cA(\omega)\}.$$ The following combines several results of Bernard.

\begin{thm}[Section 2 \cite{BernardArXiv} ]\label{thm:Bernard} Let $M$ be a closed oriented Riemannian manifold and let $\omega$ be a 1-form defining magnetic Hamiltonian $H$ and Lagrangian $L$. Then the set $\widetilde \cA(\omega)$ is nonempty, is contained in $H^{-1}(c(\omega))$, and is invariant under the Hamiltonian flow on $T^*M$. Moreover, for any $(x,p)\in \widetilde\cA(\omega)$ and any $C^1$ critical subsolution $u:M\to\R$, the trajectory $(x^t,p^t)$ of $(x,p)$ under the Hamiltonian flow is calibrated by $u$: $$u(x^T)- u(x^s) = \int_s^T c(\omega)+L(x^t,\dot x^t)dt,$$ for all $s<T$. For each such trajectory, the tangent vector $\dot x^t$ has energy level $$\frac{1}{2}|\dot x^t|^2=c(\omega).$$
\end{thm}
\begin{proof} Theorem B in \cite{BernardArXiv} gives that the Aubry set is nonempty. The invariance of $\widetilde \cA(\omega)$ is Proposition 8 in \cite{BernardArXiv}. The calibration result is the first claim on page 6 of  \cite{BernardArXiv}, and the final claim follows from Hamilton's equations and the fact $H(x^t,p^t)=c(\omega)$. In particular, $\dot x^t = \d_p H(x^t,p^t),$ and the right-hand-side is the vector musically dual to $p^t+\omega_{x^t}$. Then, we can check $$\frac{1}{2}|\dot x^t|^2 = \frac{1}{2}(p^t+\omega_{x^t})(\dot x^t) = \frac{1}{2}|p^t+\omega_{x^t}|^2=H(x^t,p^t)=c(\omega).\qedhere$$
\end{proof}

\begin{thm}\label{thm:mane-crit-speed}  Let $M$ be a closed oriented Riemannian $n$-manifold with trivial first Betti number. Let $\omega$ be a 1-form. For any $s>0,$ $$\fm_s(\omega)\leq \sqrt{2 c_0(\omega)}.$$ When $c_0(\omega)>0$, equality holds when $s=s_0=\sqrt{2c_0(\omega)}$ and the magnetic comass is realized along a magnetic geodesic.
\end{thm}
\begin{proof}
The inequality follows from a standard calculation. Let $\g$ be a speed $s$ magnetic geodesic. Then for any $C^1$ function $u:M\to \R$, we have $$\frac{1}{sT}\left|\int_0^T\omega(\dot\gamma(t))dt\right| \leq  \frac{1}{sT}\left|\int_0^T(du + \omega)(\dot\gamma(t))dt\right| + \frac{1}{sT}\left|\int_0^T du(\dot\gamma(t))dt\right|.$$
Because $u$ is bounded, the term $$\left|\int_0^T du(\dot\gamma(t))dt\right| = \left| u(\gamma(T))-u(\gamma(0))\right|$$ is bounded, so that when $T\to\infty$ that error term vanishes. Because $b_1(M)=0,$ there are no harmonic 1-forms. Choosing functions $u_i$ such that $||du_i+\omega||_\infty\to s_0$, we can estimate $$\limsup_{T\to\infty}\frac{1}{sT}\left|\int_0^T\omega(\dot\gamma(t))dt\right| \leq ||du_i+\omega||_\infty=\sqrt{2c_0(\omega)} + o(1).$$
Then taking $i\to\infty$ and using magnetic geodesic trajectories $\g_j$ approximating the magnetic comass, the claimed inequality follows.

We now turn to proving equality at the critical energy level.
Let $u:M\to\R$ be a $C^1$ critical subsolution. 
Again using that $b_1(M)=0$, we have that $c_0(\omega)=c(\omega)$. Thus, by Theorem \ref{thm:Bernard}, we can take $(x,p)\in\widetilde \cA(\omega)$ with trajectory $(x^t,p^t)$ and consider the curve $\gamma(t) = x^t$, which is calibrated by $u$. Thus, $$ \frac{1}{T}\left(u(\gamma(T))-u(\gamma(0))\right) =c(\omega) + \frac{1}{T}\int_0^TL(\gamma(t),\dot\gamma(t))dt.$$
Because $u$ is bounded, the left-hand-side vanishes as $T\to\infty.$ Additionally, since pointwise we have $$\frac{1}{2}|\dot\gamma(t)|^2=c(\omega),$$ we have that, $$ c(\omega)= o(1) - \frac{1}{T}\int_0^TL(\gamma(t),\dot\gamma(t))dt=o(1)-\frac{1}{T}\int_0^T c(\omega)-\omega(\dot\gamma(t))dt.$$
Therefore, $$2c(\omega) = \lim_{T\to\infty}\frac{1}{T}\int_0^T\omega(\dot\gamma(t))dt.$$ The trajectory $\gamma$ has constant speed $s = \sqrt{2c(\omega)}$, thus, $$\sqrt{2c(\omega)}= \frac{2c(\omega)}{s} = \lim_{T\to\infty}\frac{1}{sT}\int_0^T \omega\circ\dot\gamma(t)dt\leq\fm_s(\omega),$$ and the claim follows from the fact $c_0(\omega)=c(\omega)$.
\end{proof}

When $M$ is hyperbolic, we defined the stable isoperimetric constant to be $$\rho(M) = \inf_{\substack{\g\neq 1\\ [\g]=0}} \frac{\ell(\g)}{\scl(\g)},$$ where the infimum runs over all homotopically nontrivial rationally nullhomologous closed geodesics. 
The stable commutator length function $\scl(\g)$ is the stable genus of a surface bounding (powers of) $\g$; this is a topological analogue of the stable area function considered earlier in this paper. Specifically, for a rationally nullhomologous loop $\g,$ $$\scl(\g) = \inf\left\{\frac{\chi_-(S_k)}{2k}~:~ S_k\to M \text{ and }\d S_k = \g^k,~k>0\right\},$$ where $\chi_-(S)=\max\{0,-\chi(S)\}$ for $S $ a connected surface, and which is defined additively on components for disconnected surfaces.
See Calegari's book \cite{Calegari} for a general reference on stable commutator length.

\begin{prop}\label{prop:pre-iso-gap}
    Let $M$ be a closed hyperbolic $n$-manifold with trivial first Betti number. Let $\omega$ be a Laplace eigenform with eigenvalue $\mu^2$ normalized so that $\norm{\omega}_2=1$ and such that $d\omega$ satisfies a $C$-mean-value-inequality, then $$\rho(M)\leq \frac{4\pi C \mu}{\fm_\infty(\omega)}.$$
\end{prop}
\begin{proof}
The function $q:\pi_1M\to\R$ given by $$q(\g)= \frac{1}{2C\pi \mu}\int_{\gamma}\omega,$$ where the integral is over the unique geodesic in free homotopy class of $\g,$ is a homogeneous quasimorphism with defect at most one (see \cite{Calegari} Section 2.3.1); these are called de Rham quasimorphisms. Bavard duality (see \cite{Calegari} Section 2.5) implies $$\rho(M) = \inf_{\gamma} \frac{\ell(\gamma)}{\scl(\gamma)}\leq \inf_\gamma \frac{4\pi C \ell(\gamma)\mu}{\left|\int_{\gamma}\omega\right|} = \inf_\g\frac{4 \pi C \mu}{\frac{1}{\ell(\gamma)}\left|\int_{\gamma}\omega\right|} = \frac{4\pi C \mu}{\fm_\infty(\omega)},$$ where we used the definition of the geodesic comass in the last equality and the fact $b_1(M)=0$ ensures a sequence $\g_k$ of closed geodesics whose limiting time averages realize the comass are all rationally nullhomologous.
\end{proof}
    In dimension 3, if one uses a curl eigenform $\theta$ with eigenvalue $\mu$ instead of a Laplace eigenform, and one normalizes so that $||\theta||_\infty=1,$ one obtains the simpler estimate $$\rho(M)\leq \frac{4\pi|\mu|}{\fm_\infty(\theta)}.$$

\section{Hyperbolic geometry and magnetic geodesics}\label{sec:hyp-mag-geods}

A straightforward calculation shows that for a magnetic flow associated with an eigenform, the corresponding magnetic geodesics have geodesic curvature controlled by the eigenvalue. In particular, when the eigenvalue is small and the manifold is hyperbolic, they are quasigeodesics with small quasigeodesic constant. For high speed magnetic geodesics, such phenomena are studied in \cite{PeyerimhoffSiburg}. The utility of the following proposition comes from the precise constants and the sensitivity to the eigenvalue.

\begin{prop}\label{prop:gen-quasigeod}
    Let $\omega\in\Omega^1(M)$ be a coexact Laplace eigenform on a closed hyperbolic $n$-manifold $M$ with eigenvalue $\mu^2$ and such that $d\omega$ satisfies a $C$-mean-value-inequality. Then all of the magnetic geodesics with speed $s$ satisfying $C\mu||\omega||_2<s$ are $(1-(C\mu||\omega||_2/s)^2)^{-1/2}$-quasigeodesics.
\end{prop}

\begin{proof}
Recall that the magnetic geodesic equation for the field $\Omega=d\omega$ is $$\nabla_{\dot\g}\dot\gamma = Y(\dot\gamma),$$ where $Y$ is the Lorentz force defined via the Riemannian metric by $\Omega(u,v) = g( Y(u),v).$
    We compute $$||\Omega||_\infty=||d\omega||_\infty \leq C\mu||\omega||_2,$$ by applying the mean-value-inequality.
    For a vector $v$ with norm $s=|v|$ such that $Y(v)\neq0,$
    we can estimate $$|Y(v)|^2= g( Y(v),Y(v) ) = \Omega(v,Y(v))\leq s||d\omega||_\infty |Y(v)|.$$ Dividing both sides by $|Y(v)|$ then gives $$|Y(v)|\leq  s||d\omega||_\infty \leq sC\mu||\omega||_2. $$
    Additionally, if $Y(v)=0,$ the claimed inequality is trivial.
    The geodesic curvature of a magnetic geodesic with speed $s$  is given by $$|\kappa(\gamma(t))| = s^{-2}|\nabla_{\dot\gamma(t)}\dot\gamma(t)|$$ thus $$|\kappa(\gamma(t))|  \leq C\mu||\omega||_2/s.$$ Lemma 2.5 of \cite{Leininger} implies $\g$ is a $(1-(C\mu||\omega||_2/s)^2)^{-1/2}$-quasigeodesic as soon as $C\mu||\omega||_2/s<1.$
\end{proof}

We also state a version of this result for normalized curl eigenforms in dimension three, where we have the most control. As the proof is essentially identical, we omit it.

\begin{prop}\label{prop:quasigeod}
    Let $\omega$ be a coexact curl eigenform on a closed hyperbolic 3-manifold $M$ with eigenvalue $\mu$ satisfying $0<|\mu| < s$ normalized so that $||\omega||_\infty=1.$ Then all of the magnetic geodesics with speed $s$ are $(1-(|\mu|/s)^2)^{-1/2}$-quasigeodesics.
\end{prop}

We now show that if we have an eigenform with eigenvalue $\mu$ and we consider the integral over a quasigeodesic curve with small geodesic curvature, then that integral is comparable to the integral over the geodesic shadowed by the quasigeodesic; this then yields Theorem \ref{thm:main}.

\begin{lem}\label{lem:average-diff-bound}
Let $M$ be a closed $n$-dimensional hyperbolic manifold. 
Let $\omega\in \Omega^1(M)$ be a 1-form such that $\norm{d\omega}_\infty\leq D$ and $\norm{\omega}_\infty\leq A.$ Let $\g$ be a smooth bi-infinite parametric curve in $M$ with the absolute value of the geodesic curvature strictly less than $\kappa<1;$ this lifts to a quasigeodesic. Let $\eta$ be the hyperbolic geodesic whose lift to $\H^n$  is shadowed by a lift of $\g$.  Assume both curves are parametrized by arc length. Then $$\limsup_{T\to\infty}\frac{1}{T}\left|\int_0^T\omega(\dot\g(t))-\omega(\dot\eta(t))dt\right|\leq (A+D)\kappa.$$
\end{lem}
\begin{proof}
    We lift everything to the universal cover and compute there; we use the same Greek letters to denote the lifted curves. 
    Let $L$ be a constant bounding the Hausdorff distance between $\g$ and $\eta$, this is uniformly bounded by Lemma 2.5 in \cite{Leininger}, which also implies the curve $\gamma$ is a $(1-\kappa^2)^{-1/2}$-quasigeodesic.

    The limiting time average is independent of the initial parametrization choice of $\gamma(0)=x$ and $\eta(0)=y,$ we therefore assume $x$ and $y$ have distance at most $L.$ 
    
     For $t\in \R$, let $\alpha_t$ be the hyperbolic geodesic joining the $\eta(t)$ with $\g(t)$.     Denote by $\g_T$ the subsegment of $\g$ obtained by flowing along $\g$ for time $T;$ likewise for $\eta_T.$ Let $s(T)$ be the distance along $\eta$ of the point closest to $\g(T).$
     
    We can estimate $s(T)\leq T+2L,$ by comparing the length $s(T)$ subarc of $\eta$ with the path that traverses $\alpha_0,$ then $\gamma_T$, then the geodesic arc between $\g(T)$ and $\eta(s(T)).$ 

    Set $Q_\kappa= (1-\kappa^2)^{-1/2}$, so that $\g$ is a $Q_\kappa$-quasigeodesic.
    Next notice that because $\g$ is a $Q_\kappa$-quasigeodesic, we can estimate $d(\g(0),\g(T))\geq T/Q_\kappa.$ We can then estimate  $$T/Q_\kappa \leq 2L + s(T).$$ Combining the above, we find $|T-s(T)|\leq T-T/Q_\kappa + 2L.$

    We can therefore estimate $$\ell(\alpha_T)\leq |T-s(T)| + L \leq |T-T/Q_\kappa |+3L=(1-\sqrt{1-\kappa^2})T+3L\leq\kappa T+3L,$$ where in the final inequality, we used that $\kappa \in[0,1]$, so that $$ 1-1/Q_\kappa = 1-\sqrt{1-\kappa^2}\leq\kappa.$$ 

      We now compare the integral across $\gamma_T$ and $\eta_T$. The idea is to apply Stoke's theorem to a quadrilateral cobounding the two curves and the arcs $\alpha_0$ and $\alpha_T$, then bounding the area of this quadrilateral by $\kappa T+O(1)$. 
      
        We first set some notation.
        For a point $p\in \H^n$ and an arc $\tau$; let $F_\tau(p)$ be the union of all geodesic arcs from $p$ to $\tau;$ this is a ruled triangle. Define the rectangular cone $R_T(p)$ to be the 2-chain $$R_T(p) =  F_{\gamma_T}(p)-  F_{\alpha_T}(p) +F_{\alpha_0}(p) -F_{\eta_T}(p).$$ 
        
        We can choose $p$, depending on $T$, so that each face $F_\tau(p)$ corresponding to an arc $\tau\in\{\alpha_0,\alpha_T,\eta_T,\gamma_T\}$ is an immersed ruled triangle, as this is a generic condition.
        For  $\tau\in\{\alpha_0,\alpha_T,\eta_T\}$, it is a geodesic triangle, so has area at most $\pi.$ For the remaining face, its area can be bounded by $\kappa T+O(1)$ as follows. 
        Because $F_{\gamma_T}(p)$ is ruled, the Gauss equation in hyperbolic space implies it has Gaussian curvature at most $-1$. Denote by $\theta_i$ the intrinsic interior angles at the corners of the triangle. By Gauss-Bonnet, we can compute $$2\pi = \int_{F_{\gamma_T}(p)} K + \int_{\d F_{\gamma_T}(p)} k + \sum(\pi-\theta_i),$$ where $K$ and $k$ are the intrinsic curvatures of $F_{\gamma_T}(p)$ and $\d F_{\gamma_T}(p)$, which satisfy $K\leq-1$ and $|k| < \kappa$.
        
        We thus can estimate $$\area(F_{\gamma_T}(p))\leq  -  \int_{F_{\gamma_T}(p)} K = \sum(\pi-\theta_i) -2\pi + \int_{\d F_{\gamma_T}(p)} k \leq \kappa T +O(1),$$ where we used that $\gamma_T$ is the only side of $\d F_{\gamma_T}(p)$ with nonzero geodesic curvature and the pointwise curvature bound along $\gamma_T$.

 We can therefore compare the integral of $\omega$ over $\g_T$ and $\eta_T$ via the assumed pointwise bounds on $\omega$ and $d\omega$:
    $$\left|\int_{\gamma_T}\omega-\int_{\eta_T}\omega\right| =\left| \int_{R_T(p)}d\omega - \int_{\alpha_0} \omega +\int_{\alpha_T}\omega\right|\leq D\kappa T + A\ell(\alpha_T) + O(1).$$ 
    Dividing by $T$, using the estimate $\ell(\alpha_T)\leq \kappa T+3L$, then taking $T\to\infty$, we obtain $$\limsup_{T\to\infty}\frac{1}{T}\left|\int_0^T\omega\circ\dot\g(t)-\omega\circ\dot\eta(t)dt\right|\leq D\kappa + A\kappa.\qedhere$$ 
\end{proof}

 \begin{thm}\label{thm:pre-main}
     Let $M$ be a closed hyperbolic $n$-manifold with $b_1(M)=0$ and $\inj(M)>r.$ Let $\mu^2<1$ be the smallest eigenvalue of the Hodge Laplacian acting on coexact 1-forms; let $\omega$ be an eigenform with eigenvalue $\mu^2$ normalized so that $\norm{\omega}_2=1$. Assume $C_r\mu/\sqrt{2c_0(\omega)}<\e<1.$ Then $$\frac{\rho(M)}{4\pi C_r }\left(\frac{1-C_r\sqrt{\vol(M)}(1+\mu)\e}{ \sqrt{\vol(M)}}\right)\leq \mu,$$ where $C_r$ depends only on $r$ and $n$ (indeed, it comes from the mean-value-inequality in Proposition \ref{prop:MV}). 
 \end{thm}
 
\begin{proof}
Set $s_0=\sqrt{2c_0(\omega)}.$
Because $C_r\mu<s_0$ and magnetic geodesics are smooth, we can apply Lemma \ref{lem:average-diff-bound} and compare $\fm_{s_0}(\omega)$ and $\fm_\infty(\omega),$ using the quasigeodesicity coming from Proposition \ref{prop:gen-quasigeod}. Let $\g$ be a calibrated magnetic geodesic trajectory with speed $s_0$ realizing the magnetic comass $\fm_{s_0}(\omega)$ given by Theorem \ref{thm:mane-crit-speed}. Let $\eta$ be the geodesic shadowed by the $\gamma,$ parametrized with speed $s_0.$

Lemma \ref{lem:average-diff-bound} and arc-length reparameterization implies $$\limsup_{T\to\infty}\frac{1}{s_0T}\left|\int_0^T\omega\circ\dot\g(t)-\omega\circ\dot\eta(t)dt\right|\leq C_r (1+\mu)\e.$$
We can therefore compute \begin{align*}
    \fm_{s_0}(\omega)&=\limsup_{T\to\infty}\frac{1}{s_0T} \left|\int_0^T\omega\circ\dot\gamma(t)dt\right| \\
    &=\limsup_{T\to\infty}\frac{1}{s_0T}\left|\int_0^T\omega\circ\dot\g(t)-\omega\circ\dot\eta(t)+\omega\circ\dot\eta(t) dt\right|\\
    &\leq\limsup_{T\to\infty}\frac{1}{s_0T}\left|\int_0^T\omega\circ\dot\g(t)-\omega\circ\dot\eta(t)dt\right|+ \frac{1}{s_0T}\left|\int_0^T\omega\circ \dot\eta(t) dt\right|\\
    &\leq C_r (1+\mu)\e  + \fm_\infty(\omega).
\end{align*}

Therefore $$  \fm_{s_0}(\omega)- C_r(1+\mu)\e\leq \fm_{\infty}(\omega).$$
If $\fm_{s_0}(\omega)-C_r(1+\mu)\e\leq0,$ then the left-hand-side of the inequality in the theorem is non-positive and the claim is trivially true.
We therefore assume that the left-hand-side is positive, then by
plugging this into the estimate of Proposition \ref{prop:pre-iso-gap} gives $$\rho(M)\leq \frac{4\pi C_r \mu}{\fm_{\infty}(\omega)}\leq \frac{4\pi C_r \mu}{\fm_{s_0}(\omega)-C_r(1+\mu)\e}.$$

By Lemma \ref{lem:norm-comp}, Theorem \ref{thm:mane-crit-speed}, and the fact $\norm{\omega}_2=1$, we have $$ \frac{1}{\sqrt{\vol(M)}}\leq \fm_{s_0}(\omega)=s_0.$$ If $\vol(M)^{-1/2} - C_r(1+\mu)\e \leq 0,$ the claim is again trivial, so we assume otherwise. We then obtain $$\rho(M) \leq \frac{4\pi C_r \mu}{\frac{1}{\sqrt{\vol(M)}}-C_r(1+\mu)\e},$$ from which the claim follows.
\end{proof}

The main theorem now follows from book-keeping.

\begin{thm}[cf. Theorem \ref{thm:intro}]\label{thm:main}
         Let $M$ be a closed hyperbolic $n$-manifold with $b_1(M)=0$ and $\inj(M)>r.$ Let $\mu^2<1$ be the smallest eigenvalue of the Hodge Laplacian acting on coexact 1-forms. Then either $$\frac{1}{8C_r^2\vol(M)}\leq \mu$$ or else
         $$\frac{\rho(M)}{8\pi C_r\sqrt{\vol(M)}}\leq\mu,$$ where $C_r$ is the constant from the mean-value-inequality.
\end{thm}
\begin{proof}
Let $\omega$ be an eigenform with eigenvalue $\mu^2$ normalized so that $\norm{\omega}_2=1.$ By Lemma \ref{lem:norm-comp}, we have $$\frac{1}{2}\vol(M)^{-1/2}< \sqrt{2c_0(\omega)}=:s_0.$$ Assume $$\mu\leq \frac{1}{8C_r^2}\vol(M)^{-1}.$$ Then it follows that $$C_r\mu/s_0<C_r\frac{\frac{1}{8C_r^2}\vol(M)^{-1}}{\frac{1}{2}\vol(M)^{-1/2}}\leq \frac{1}{4C_r\sqrt{\vol(M)}}.$$ 
    Notice that $(4C_r\sqrt{\vol(M)})^{-1}< 1$, since we can scale $\omega$ to obtain a form $\eta$ with $||\eta||_\infty=1$, then we can estimate $$1=\norm{\eta}_\infty\leq C_r\norm{\eta}_2\leq C_r\sqrt{\vol(M)}||\eta||_\infty = C_r\sqrt{\vol(M)}.$$
    We therefore can apply Theorem \ref{thm:pre-main} with $\e=\frac{1}{4C_r\sqrt{\vol(M)}}<1$ to obtain $$\frac{\rho(M)}{4\pi C_r }\left(\frac{1-C_r\sqrt{\vol(M)}(1+\mu)\frac{1}{4C_r\sqrt{\vol(M)}}}{ \sqrt{\vol(M)}}\right)\leq \mu.$$

    Simplifying, we find $$\frac{\rho(M)}{4\pi C_r \sqrt{\vol(M)}}\left(1-(1+\mu)/4\right)\leq\mu.$$ Since $\mu<1$, the claim follows. 
\end{proof}

\section{Universal critical speed and the spectral gap}\label{sec:uni-crit-vel}

Here we note some interesting connections between curl eigenforms and Mañé's critical values. The starting point is the following theorem of Contreras, which requires the Mañé's universal critical value: $$\tilde c(\omega) = \frac{1}{2}\inf_{u: \widetilde M\to\R} ||\pi^*\omega+du ||_\infty^2$$ where  $\pi:\widetilde M\to M$ is the universal covering map.
The universal critical speed $\tilde s$ is then $\tilde s = \sqrt{2\tilde c(\omega)}.$ The dynamics of the magnetic geodesic flows undergoes a dramatic change at energy level $\tilde c(\omega).$
Note that there is a weighted isoperimetric interpretation of this critical value, which is related to the degree one Cheeger constant, due to Burns-Paternain \cite{BurnsPaternain}.

\begin{thm}[Theorem D \cite{Contreras}]\label{thm:subcritnull}
    Let $M$ be a closed oriented Riemannian manifold and $\omega$ a coexact form defining a magnetic field. Then for almost every speed $s$ less than the universal critical speed $\sqrt{2\tilde c(\omega)}$, there is a contractible periodic orbit with speed $s$.
\end{thm}

This, combined with Proposition \ref{prop:quasigeod} implies a relationship between eigenvalues of a curl eigenform and the corresponding universal critical value. 

\begin{prop}\label{lem:mane-eigval-ineq} Let $M$ be a closed hyperbolic 3-manifold and let $\omega$ be a coexact curl eigenform with eigenvalue $\mu$ normalized so that $\norm{\omega}_\infty=1$. Then $\sqrt{2\tilde c(\omega)}\leq |\mu| $. 
\end{prop}
\begin{proof}
    Suppose $|\mu|<\tilde s = \sqrt{2\tilde c(\omega)}$, then there is an open interval $ J = (|\mu|,\tilde s)$ such that the magnetic geodesics with speed $s \in J$ are all quasigeodesics by Proposition \ref{prop:quasigeod}. But by Theorem \ref{thm:subcritnull}, there is some speed $s$ in this interval such that there exists a contractible magnetic geodesic with that speed. Since quasigeodesics are not contractible, we have a contradiction.
\end{proof}

    There can be a gap between Mañé's universal  critical speed and the smallest eigenvalue. In particular, for a normalized curl eigenform $\omega$ with $\norm{\omega}_\infty=1$, the variational definition of the critical speed implies $\sqrt{2\tilde c(\omega)} \leq 1$. But there are examples of closed hyperbolic 3-manifolds with smallest curl eigenvalue having absolute value greater than $\sqrt{2}$, see \cite{LinLipnowski} for instance.

We end with some examples which use the above inequality to contrast the behavior of the strict and universal critical values.

\begin{prop}\label{prop:examples}
    There exists a sequence $M_n$ of closed hyperbolic 3-manifolds with $\inj(M_n)>r$ and volume $V_n=\vol(M_n)\to\infty$ along with curl eigenforms $\omega_n\in\Omega^1(M_n)$ normalized so that $||\omega_n||_\infty=1$ and such that for constants $c,C>0$  the following holds $$\frac{\tilde c(\omega_n)}{c_0(\omega_n)}\leq CV_n^3e^{-2cV_n}.$$
\end{prop}
\begin{proof}
    The examples come from Theorem C in \cite{Rudd}, all of which have trivial first Betti number, and which have smallest curl eigenvalues $\mu_n$ satisfying $|\mu_n|\leq CV_ne^{-cV_n},$ as the curl eigenvalues are square roots of the Hodge Laplacian coexact eigenvalues and the curl eigenforms are coexact Hodge Laplacian eigenforms with eigenvalues $\mu_n^2$. The corresponding eigenforms $\omega_n$ (and their differentials) satisfy a uniform $C$-mean-value-inequality, by Proposition \ref{prop:MV} for some $C$ and the controlled geometry of the manifolds $M_n$. Thus $(2C^2V_n)^{-1}\leq c_0(\omega_n)$ by Lemma \ref{lem:norm-comp} and the mean-value-inequality. After renaming constants, the claim follows from Lemma \ref{lem:mane-eigval-ineq}.
\end{proof}

\section{Explicit constants in dimension 3}\label{sec:appendix}
In this section, we derive an explicit mean-value-inequality for eigenforms with bounded eigenvalue in hyperbolic manifolds of dimension 3.
Note that the scalar Laplacian and the degree zero Hodge Laplacian differ by a sign. In this section, we use $\Delta$ to refer to the scalar Laplacian $\Delta = \text{Div} \circ\nabla = -d^*d= -\Delta_0$ acting on smooth functions, where $\Delta_0$ is the degree zero Hodge Laplacian. For the Laplacian on 1-forms, we will use a subscript indicating the degree, $\Delta_1:\Omega^1(M)\to\Omega^1(M).$

The key to deriving the explicit mean-value-inequality for 1-forms is the following weighted mean-value-inequality for subsolutions of a related scalar equation.

\begin{thm}[Theorem 3.1 \cite{LiNguyen}]\label{thm:H-MVT} Let $a:\H^n\to\R$ be a smooth function. Let $u:\H^n\to\R$ be a smooth function such that $\Delta u + au\geq 0.$ Then for any ball $B_R(x)$, the following holds: $$u(x)\leq \frac{1}{\vol(B_R(x))}\int_{B_R(x)} u(y) w(x,R;y)d\vol(y),$$ where $$w(x,R;y)=1+a(y)\int_{d(x,y)}^R(\sinh \tau)^{n-1}\int_{d(x,y)}^\tau \frac{1}{(\sinh\xi)^{n-1}}d\xi d\tau.$$ 
\end{thm}
The basic observation is that the norm squared function $u(x) = |\omega|^2_x$ for an eigenform $\omega$ is a subsolution to such an equation. 

\begin{lem}\label{lem:subharm}
 Let $M$ be a closed hyperbolic $n$-manifold. Let $\omega$ be a 1-form such that $\Delta_1 \omega=\lambda\omega$ and assume $\lambda\leq \L.$  Consider the function $u(x)=|\omega|_x^2,$ then, $$\Delta u+ (2\L + 2n -2)u\geq 0.$$ In particular, Theorem \ref{thm:H-MVT} applies to $u$ with $a(x) = (2\L + 2n -2)$.
\end{lem}
\begin{proof}
Recall the Bochner formula $$\Delta u = 2|\nabla \omega|^2-2(n-1+\lambda)u.$$ Because $u\geq 0 $ and $\lambda\leq \L$, we rearrange and find that $$\Delta u+ (2\L + 2n -2)u\geq\Delta u+(2\lambda + 2n -2)u=   2|\nabla \omega|^2\geq  0.\qedhere$$
\end{proof}
We can therefore apply Theorem \ref{thm:H-MVT} to obtain a weighted mean value inequality for $\omega$; we write this as a separate lemma for convenience. To simplify notation, set $$V_R:=\vol(B_R(x)).$$

\begin{lem}
    Let $M$ be a closed hyperbolic $n$-manifold with injectivity radius at least $R>0.$ Let $\omega$ be a 1-form with $\Delta_1 \omega=\lambda\omega$ and assume $\lambda\leq \L.$  Consider the function $u(x)=|\omega|_x^2,$ then, $$u(x)\leq \frac{1}{V_R}\int_{B_R(x)} u(y)w(x,R;y)d\vol(y),$$ where $w$ is the function from Theorem \ref{thm:H-MVT}, given by
    
    $$w(x,R;y)=1+(2\L + 2n -2)\int_{d(x,y)}^R(\sinh \tau)^{n-1}\int_{d(x,y)}^\tau \frac{1}{(\sinh\xi)^{n-1}}d\xi d\tau.$$ 
\end{lem}
\begin{proof}
    This is immediate from Theorem \ref{thm:H-MVT} and Lemma \ref{lem:subharm}.
\end{proof}

We restrict to 3-manifolds in what follows. The key is to control the $L^2$ norm of $w(x,R;y)$. Fix $R>0$,
consider the function $$ W(s):=\int_{s}^R(\sinh \tau)^{2}\int_{s}^\tau \frac{1}{(\sinh\xi)^{2}}d\xi d\tau.$$  Note that this depends on $R,$ but we suppress that from the function notation as $R$ is fixed. 
Notice that for $x\in\H^3$ and $R>0$, and given any $y\in B_R(x)$ with $s=d(x,y)$, we have  $$w(x,R;y) = 1+(2\L+4)W(s).$$

\begin{lem}\label{lem:explicit}
   Fix a point $x\in \H^3$ and let $y\in B_R(x)$.  Consider the function $W$ above. Then     $$\int_0^RW(s)\sinh(s)^2ds=\frac{1}{8} (R (2 + \cosh(2 R)) - 3 \cosh(R) \sinh(R))$$ and  
\begin{align*}
    \int_{0}^RW(s)^2\sinh(s)^2ds&= \frac{(8R^3 - 12(R^2 - 2)\sinh(2R) + 3\sinh(4R) - 24R\cosh(2R)-36R)}{192}.
\end{align*}
\end{lem}
\begin{proof}

Let $0<s\leq\tau\leq R$. First observe that
$$\int_s^\tau \frac{1}{(\sinh\xi)^2}d\xi =\coth(s)-\coth(\tau). $$
Next, we compute the integral $$W(s) = \int_s^R\sinh(\tau)^2(\coth(s)-\coth(\tau))d\tau,$$ and obtain $$W(s) = \frac{1}{4}(-\cosh(2 R) + \cosh(2 s) + \coth(s) (-2 R + 2 s + \sinh(2 R) - \sinh(2 s))).$$

We now will compute $$C_1 = \int_0^RW(s)\sinh(s)^2ds$$ and $$C_2= \int_0^RW(s)^2\sinh(s)^2ds.$$
Using {\tt sage} to evaluate these integrals, and {\tt{mathematica}} to simplify the results, we obtain:
\begin{align*}
    C_1 &=\frac{1}{8} (R (2 + \cosh(2 R)) - 3 \cosh(R) \sinh(R)).
\end{align*}

and
\begin{align*}
    C_2=\frac{(8R^3 - 12(R^2 - 2)\sinh(2R) + 3\sinh(4R) - 24R\cosh(2R)-36R)}{192} .
\end{align*}

\end{proof}
\begin{lem}\label{lem:explicitl2bound} The function $w(y) = w(x,R;y)$ has $L^2$ norm in the ball $B= B_R(x)$ satisfying
    $$||w(x,R;\cdot)||_{L^2(B)} \leq C_{R,\L}$$ where $$C_{R,\L}^2=V_R+(2\L+4)^24\pi  \int_0^RW(s)^2\sinh^2(s)ds + (4\L+8)4\pi \int_{0}^RW(s)\sinh^2(s)ds,$$ where explicit expressions for the integrals are given in Lemma \ref{lem:explicit}.
\end{lem}
\begin{proof}
    Recall that $w(x,R;y)^2 = 1+(2\L+4)^2W(d(x,y))^2 + (4\L+8)W(d(x,y))$; the claim follows now from the previous two lemmas and the expression of the $L^2$ norm in the ball using polar coordinates.
\end{proof}

We now prove the desired mean value inequality. 

\begin{thm}\label{thm:appendix}
    Let $M$ be a closed hyperbolic $3$-manifold with injectivity radius at least $R>0.$ Let $\omega$ be a 1-form with $\Delta_1 \omega=\lambda\omega$ and assume $\lambda\leq \L.$  Then, $$||\omega||_\infty\leq \frac{C_{R,\L}}{V_R}||\omega||_2$$ where $C_{R,\L}$ is the constant from Lemma \ref{lem:explicitl2bound}.
\end{thm}
\begin{proof}
    With $u(x) = |\omega|^2_x$, take $x$ realizing the maximum and consider a ball $B$ of radius $R$ centered at $x$. Because $R<\inj(M)$, we can identify $B$ with a ball in $\H^3$. Then we have $$||\omega||_\infty^2 = u(x)\leq \frac{1}{V_R}\int_{B}u(y)w(x,R;y)d\vol(y).$$ 
    Using Cauchy-Schwarz, we estimate $$\frac{1}{V_R}\int_{B}u(y)w(x,R;y)\leq \frac{1}{V_R}\left(\int_{B}u^2\right)^{1/2}\left(\int_{B}w(x,R;y)^2\right)^{1/2}.$$
    Then, we have $$||\omega||_\infty^2\leq \frac{||w(x,R;\cdot)||_{L^2(B)}}{V_R}||u||_{L^2(B)}\leq \frac{C_{R,\L}}{V_R}||u||_{L^2(B)}.$$
    Notice that since $u = |\omega|^2,$ we have $$||u||_{L^2(B)}=||\omega||_{L^4(B)}^2.$$
    We then estimate $$||\omega||_{L^4(B)}^2\leq ||\omega||_\infty||\omega||_{L^2(B)},$$
    which then gives,
    $$||\omega||_\infty^2\leq \frac{C_{R,\L}}{V_R}||\omega||_\infty||\omega||_{L^2(B)}.$$
Since $||\omega||_{L^2(B)}\leq||\omega||_2$, the claim follows.
\end{proof}
To end, we include an explicit estimate in the special case where $R=1$ and $\L=1,$ which implies Theorem \ref{thm:intro3man}  and Theorem \ref{thm:intro:area-hyp} when combined with our main results.
\begin{cor}
    Let $M$ be a closed hyperbolic 3-manifold with $\inj(M)>1$ and let $\omega$ be a coexact eigen 1-form with eigenvalue $\lambda$ at most $1.$ Then, $$ ||\omega||_\infty\leq 1.04||\omega||_2. $$
\end{cor}

\footnotesize{
\bibliographystyle{alpha}
\bibliography{bib}}

\end{document}